\documentclass[10pt, reqno]{amsart} 

\usepackage[margin=3.5 cm]{geometry} 
\usepackage{amsmath,amssymb,amsthm}   

 \usepackage{url}     
 
 \usepackage{hyperref}

\usepackage{xy}
\input xy 
\xyoption{all}

 \usepackage{graphicx} 
\usepackage{pstricks}
\usepackage{epstopdf}

\newcommand{\bR}{{\mathbb R}}

\newcommand{\bZ}{{\mathbb Z}}

\newcommand{\hal}{\mbox{\boldmath{$\alpha$}}}
\newcommand{\hbe}{\mbox{\boldmath{$\beta$}}}

\newcommand{\fs}{\mathfrak{s}}

\newcommand{\De}{\Delta}

\newcommand{\om}{\omega}

\newcommand{\Si}{\Sigma}

\newcommand{\be}{\beta}
\newcommand{\al}{\alpha}

\newcommand{\si}{\sigma}

\newcommand{\ga}{\gamma}

\providecommand{\abs}[1]{\lvert#1\rvert}

\newcommand{\bdd}{\partial}

\newtheorem{thm}{Theorem} % [section]
\newtheorem{theorem/definition}{Theorem/Definition}[section]
\newtheorem{prop}[thm]{Proposition}

\newtheorem{lemma}[thm]{Lemma}

\newtheorem{prob}{Problem}

\theoremstyle{definition}

\newtheorem{deff}{Definition}

\theoremstyle{remark}

\newtheorem{rmk}{Remark}

\DeclareMathOperator{\Int} {Int}

\DeclareMathOperator{\Spin} {Spin}

\begin{document} 

%TITLE
\title{Sutured Floer homology distinguishes between Seifert surfaces}         
\date{\today}       
\author{Irida Altman}
\maketitle

%ABSTRACT
\begin{abstract}  
We exhibit the first example of a knot $K$ in the three-sphere with a pair of minimal genus Seifert surfaces $R_1$ and $R_2$ that can be distinguished using the sutured Floer homology of their complementary manifolds together with the $\Spin^c$-grading.  This answers a question of Juh\'asz.  More precisely, we show that the Euler characteristic of the sutured Floer homology distinguishes between $R_1$ and $R_2$, as does the sutured Floer polytope introduced by Juh\'asz.  Actually, we exhibit an infinite family of knots with pairs of Seifert surfaces that can be distinguished by the Euler characteristic.

\end{abstract}

%---------DOCUMENT BEGINNING--------------

\section{Introduction}

Let $K$ be an oriented knot in the three-sphere $S^3$.  Then $K$ is the oriented boundary of at least one  connected compact oriented surface in $S^3$ called a Seifert surface for $K$.  Two Seifert surfaces $R_1$ and $R_2$ of a knot are considered to be {\it equivalent} if they are ambient isotopic in the knot complement.  There are a number of invariants that provide obstructions to two Seifert surfaces being equivalent; possibly the first two that come to mind are the genus of the surface and the fundamental group of the surface complement.  In general, any invariant of the surface complement offers an obstruction to the equivalence of $R_1$ and $R_2$.  Given a Seifert surface $R$, the complement $S^3(R):=\nolinebreak S^3 \setminus \Int (R \times I)$ together with the curve $\bdd R \times \{1/2\}$ on the boundary is a type of 3-manifold called a {\it balanced sutured manifold}.  Therefore, it is reasonable to investigate the possibility of using {\it sutured Floer homology}, an invariant of balanced sutured manifolds introduced by Juh\'asz \cite{Ju06}, to distinguish between equivalence classes of Seifert surfaces.  

Sutured Floer homology associates to a given balanced sutured manifold $(M,\ga)$ a finitely generated bigraded abelian group denoted by $SFH(M,\ga)$.  The group $SFH(M,\ga)$ is graded by the relative $\Spin^c$ structures $\fs \in \Spin^c(M,\ga)$, and has a relative $\bZ_2$ grading.  The support of sutured Floer homology gives rise to the {\it sutured Floer polytope}  $P(M,\ga)$, defined in \cite{Ju10a}, which is a polytope in $H^2(M, \bdd M;\bR)$.  

Suppose $R$ is any minimal genus Seifert surface for a knot in $S^3$.  Then Juh\'asz showed that $SFH(S^3(R))$ is a knot invariant  \cite{Ju08}; that is,  the top term of {\it knot Floer homology} \cite{OS04b, Ra03} is isomorphic to the sutured Floer homology of the complement:
\[ 
SFH(S^3(R)) \cong \widehat{HFK}(K,\mbox{genus}(R)).
\]
As the isomorphism is in terms of ungraded abelian groups, it is interesting to ask whether the extra structure, given by the $\Spin^c$ grading of $SFH(S^3(R))$, enables sutured Floer homology to distinguish between two minimal genus Seifert surfaces.  

\begin{prob} \cite[Problem 2]{Juprob} \label{prob}
Is there a knot $K$ in $S^3$ that has two minimal genus Seifert surfaces $R_1$ and $R_2$ that can be distinguished using $SFH(S^3(R_i))$ together with the $\Spin^c$-grading?  Is there an example where the sutured Floer homology polytopes of $S^3(R_1)$ and $S^3(R_2)$ are different?
\end{prob}

Until now research has provided evidence to suggest that the answer to both question is no.  For example, the first obvious place to investigate these ideas are small knots.  Indeed, in  \cite[Ex.\,8.6]{FJR10} the authors compute $SFH(S^3(R))$ for $R$ ranging through the minimal genus Seifert surfaces for knots with less than 10 crossings.  These small knots have either a unique minimal genus Seifert surface, or all of their minimal genus Seifert surfaces can be identified with Murasugi sums of bands.  However, Juh\'asz showed that the sutured Floer homology of the complement of a Murasugi sum is the tensor product of the sutured Floer homology of the complement of each summand \cite[Cor.\,8.8]{Ju08}.  It is immediate from \cite[Prop.\,5.4]{Ju10a} that the relative $\Spin^c$ grading of the tensor product is independent of how the surfaces were summed.   Thus, all surfaces arising from Murasugi sums (and even dual Murasugi sums) of the same summands cannot be distinguished even by the $\Spin^c$-graded sutured Floer homology group. 

The aim of this article is to give an affirmative answer to both questions posed in Problem \ref{prob} by exhibiting examples of the phenomena.  Our examples come from a family of knots that were studied by Lyon \cite{Lyon}; see Figure \ref{Lyons figure}.  Indeed, we show that even the Euler characteristic $\chi SFH$ of sutured Floer homology  distinguishes between two Seifert surfaces for each of these knots.  

\begin{thm} \label{thm}
There are infinitely many knots with the property that each knot has two minimal genus Seifert surfaces $R_1$ and $R_2$ such that
\[
\chi SFH(S^3(R_1)) \not \sim \chi SFH(S^3(R_2)).
\]
Moreover, for at least one of these knots the sutured Floer polytopes $P(S^3(R_1))$ and $P(S^3(R_2))$ are such that there exists no affine isomorphism of $H^2(M,\bdd M;\bR)$ taking one polytope to the other.
\end{thm}

Here the symbol `$\not \sim$' is used to mean the negation of an appropriate equivalence relation (see end of Section \ref{prelims}). 

\begin{figure}[h]
\centering
\includegraphics [scale=0.5]{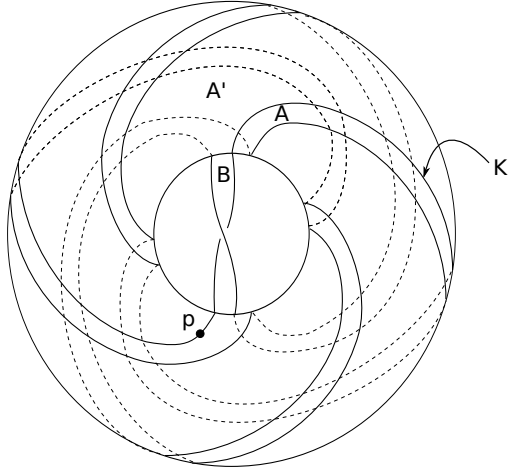}
\caption{One of the knots studied by Lyon \cite[Fig.\,1]{Lyon}.}
\label{Lyons figure}
\end{figure}

We prove the first statement of Theorem \ref{thm} by applying the work of Friedl, Juh\'asz and Rasmussen in \cite{FJR10}, where they give a way of finding the Euler characteristic using Fox calculus. Let $(M,\ga)$ be a balanced sutured manifold.  Then after an identification of $\Spin^c(M,\ga)$ with $H_1(M;\bZ)$ (see subsection 2.2 for more details), the Euler characteristic $\chi SFH(M,\ga)$ can be identified with a type of Turaev torsion polynomial denote by $\tau(M,\ga)$ \cite[Sec.\,3]{FJR10}.  Here the {\it sutured torsion} $\tau(M,\ga)$ is a well-defined element of the group ring $\bZ[H_1(M)]$ up to multiplication by units of the group ring. The sutured torsion has similar properties to that of the classical Alexander polynomial, and so $\tau(M,\ga)$ can be thought of as the generalisation of the Alexander polynomial to sutured manifolds.

For the second statement of Theorem \ref{thm}, we compute the sutured Floer homology and polytopes for one particular knot and two of its Seifert surfaces. Firstly, for each knot that we study, the considered Seifert surfaces $R_1$ and $R_2$ are disjoint.  Moreover, the two sutured manifolds, $X$ and $Y$, obtained by cutting the knot exterior along $R_1$ and $R_2$ are handlebodies of genus two.  Secondly, we are able to find a particular knot, for which there is disk decomposition of $X$ and $Y$ along a product disk.  The latter gives us a way of explicitly computing the sutured Floer homology groups $SFH(X)$ and $SFH(Y)$ (see Propositions \ref{product decomp} and \ref{torus prop}).  Then the groups $SFH(S^3(R_1))$ and $SFH(S^3(R_2))$ are the tensor product $SFH(X) \otimes SFH(Y)$ \cite[Prop.\,8.6]{Ju08}, where the $\Spin^c$ grading can be derived from the appropriate Mayer-Vietoris maps on the level of the first homology groups \cite[Prop.\,5.4]{Ju10a}.  Observe that this method allows us to compute the top term of knot Floer homology of a rather complicated knot --- something which would be significantly harder to do directly.

Lastly, it is important to note that in order to solve Problem \ref{prob}, in Theorem \ref{thm} we use only the $\Spin^c$ grading of the sutured Floer homology.  Therefore, our result is independent of any auxiliary data, in comparison to the work of Hedden, Juh\'asz, and Sarkar  \cite{HJS08}, who show the nonequivalence of two Seifert surfaces $R_1$ and $R_2$ of the knot $8_3$ using sutured Floer homology methods together with properties of the Seifert form. They prove that there is no map $\sigma \colon \Spin^c(S^3(R_1)) \to \Spin^c(S^3(R_2))$ that induces an isomorphism of sutured Floer homology groups for every relative $\Spin^c$ structure, and that is compatible with an isomorphism  $H_1(S^3(R_1)) \to H_1(S^3(R_2))$ which preserves the Seifert form.   
 
Section \ref{prelims} covers some preliminary definitions and explains the method for computing the Euler characteristic via Fox calculus.  Section \ref{example} contains the computations and proof of Theorem \ref{thm}.
\\ \\
\noindent {\it Acknowledgements.} I would like to thank my Ph.D. adviser Stefan Friedl  for many helpful discussions and suggestions.  I am very grateful to Andr\'as Juh\'asz for highlighting interesting questions related to sutured Floer homology, and for pointing out errors in earlier drafts of this work.  I thank Saul Schleimer for stimulating conversations and his interest in my work. 

I owe much gratitude to the University of Warwick and the Warwick Mathematics Institute for generously supporting me through a Warwick Postgraduate Research Scholarship.

%%%%%%%%%%%%%%%%%%%%%%%%%%%%%%%%%%%%%%%%
%%%%%%%%%%%%%%%%%%%%%%%%%%%%%%%%%%%%%%%%
%%%%%%%%%%%%%%%%%%%%%%%%%%%%%%%%%%%%%%%%
\section{Preliminaries} \label{prelims}

To begin with, let us set up some conventions.  Given a space $X$, we denote by $H_*(X)$ the homology group with integer coefficients $H_*(X;\bZ)$.  Further, we write $\chi(X)$ to mean the Euler characteristic $\chi(H_*(X))$.  Lastly, for $K$ a submanifold of $M$ denote by $N(K)$ a regular neighbourhood of $K$ in $M$.

%%%%%%%%%%%%%%%%%%%%%%%%%%%%
\subsection{Sutured manifolds}
The notion of a sutured manifold $(M,\ga)$ was first defined by Gabai \cite{Gabai}.  Here we give a less general definition that is suited to thinking about a particular class of so-called {\it balanced} sutured manifolds defined by Juh\'asz \cite{Ju06}.  

\begin{deff} 
A {\it sutured manifold} $(M,\ga)$ is a compact oriented 3-manifold $M$ with boundary, together with a set $s(\ga)$ of oriented and pairwise disjoint simple closed curves in $\bdd M$ called {\it sutures}, which satisfy two conditions.
The first condition is that each component of $\bdd M$ must contain at least one suture.  Fix a neighbourhood $\ga$ of the sutures in $\bdd M$ that consists of a pairwise disjoint collection of annuli.  The second condition is that every component $R$ of the surface $\bdd M \setminus \Int(\ga)$ must be orientable in such a way that the induced orientation on each component of $\bdd R$ represents the same homology class as the corresponding suture in $H_1(\ga)$.
\end{deff}

Let $R(\ga)$ be the exterior of the sutures in the boundary of $M$, that is, $R(\ga):=\bdd M \setminus \Int (\ga)$.  Now each component of $R(\ga)$ has two orientations: one induced by the orientation of $M$, and one compatible with the orientation of the sutures.  Denote by $R_+(\ga)$ the set of components of $R(\ga)$ on which the two orientations match, and denote by $R_-(\ga)$ the set of remaining components.

\begin{deff}
A sutured manifold $(M,\ga)$ is said to be {\it balanced} if it has no closed components and if there is an equality of Euler characteristics $\chi (R_+(\ga))=\chi (R_-(\ga))$.
\end{deff}

\begin{rmk}
Our definition of a balanced sutured manifold is equivalent to that of Juh\'asz \cite[Def.\,2.1]{Ju06}.
\end{rmk}

In particular, given a Seifert surface $R$, the complement $S^3(R)$ is a balanced sutured manifold with a single suture $s(\ga):=\bdd R \times \{\frac{1}{2}\}$ and a single annular neighbourhood $\ga:=\bdd R \times I$. We refer to $(S^3(R),\ga)$ as the sutured manifold {\it complementary} to $R$.  Actually, since $R_+(\ga)$ consists of one component only, $S^3(R)$ is {\it strongly balanced} \cite[Def.\,3.5]{Ju08}.   A balanced sutured manifold $(M,\ga)$ is  strongly balanced if  for each component $F$ of $\bdd M$, we have the equality $\chi (F \cap R_+(\ga)) = \chi (F \cap R_-(\ga))$ \cite[Def.\,3.5]{Ju08}.  The fact that $S^3(R)$ is strongly balanced becomes relevant later, as the sutured Floer polytope is only defined for strongly balanced sutured manifolds.

Next, we describe an operation on sutured manifolds that leaves the sutured Floer homology unchanged. Suppose $(M,\ga)$ is a balanced sutured manifold, and $D$ is a properly embedded disc in $M$, such that $\abs{D \cap s(\ga)}=2$ and $\bdd D \cap \ga$ consists of essential arcs.  Choose a regular neighbourhood $N(D):=D \times [0,1]$ such that $\bdd D \times [0,1] \subset \bdd M$.  Denote by $D_+:=D \times \{0\}$ and $D_-:= D \times \{1\}$.  Then the {\it product decomposition} of $(M,\ga)$ along $D$ is an operation on $M$ which results in another balanced sutured manifold $(M',\ga')$ defined by 
\begin{gather*}
M':=M \setminus D \times (0,1), \\
\ga':=(\ga \cap M) \cup \left(N(D_+)\cap R_-(\ga)\right) \cup \left(N(D_-) \cap R_+(\ga)\right).
\end{gather*}
We use product decomposition in the proof of Theorem \ref{thm}, and we denote it by 
\[
(M,\ga) \leadsto^D (M',\ga').
\] 
 \begin{prop} \cite[Lemma\,9.13]{Ju06} \label{product decomp}
 Suppose $(M,\ga)$ is a balanced sutured manifold, and there is a product decomposition $(M,\ga) \leadsto^D (M',\ga')$.  Then $SFH(M,\ga)=SFH(M',\ga')$.
 \end{prop}

Product decomposition is a useful operation when computing the sutured Floer homology of a specific sutured manifold.  In particular, in the proof of Theorem \ref{thm}, we have handlebodies of genus two with a single suture, and each of the handlebodies can be product decomposed into a solid torus with two sutures on the boundary. The sutured Floer homology of $S^1 \times D^2$ with any collection of sutures is already known; see Proposition \ref{torus prop}.

%%%%%%%%%%%%%%%%%%%%%%%%%%%%
\subsection{Relative $\Spin^c$ structures and the sutured Floer polytope} \label{subsec polytope}

Every balanced sutured manifold $(M,\ga)$ has an associated space of relative $\Spin^c$ structures $\Spin^c(M,\ga)$; we define relative $\Spin^c$ structures in the following paragraph.  For each $\fs \in \Spin^c(M,\ga)$ there is a well-defined abelian group $SFH(M,\ga,\fs)$ \cite{Ju06}, and the direct sum of these groups forms the {\it sutured Floer homology} of $(M,\ga)$.  That is,
\[
SFH(M,\ga)=\bigoplus_{\fs \in \Spin^c(M,\ga)} SFH(M,\ga,\fs).
\]
Juh\'asz computed the sutured Floer homology of $(M,\ga)$ when $M$ is the solid torus.  We use this in the proof of Theorem \ref{thm} to compute the polytopes.  Let $T(p,q;n)$ be the balanced sutured manifold $(M,\ga)$, where $M$ is a solid torus, and the sutures are $n$ parallel $(p,q)$ torus knots.  Here $p$ denotes the number of times the curve on $\bdd M$ goes around in the longitudinal direction.  Note that $n$ has to be even.

\begin{prop} \cite[Prop.\,9.1]{Ju10a} \label{torus prop}
Suppose that $T(p,q;n)$ is as described above, and suppose that $n=2k+2$, for some nonnegative integer $k$.  Then there is an identification 
\[
\Spin^c(T(p,q;n)) \cong \bZ
\]
such that the following holds
\[
SFH(T(p,q;n),i) \cong 
\begin{cases}
\bZ^{\binom{k}{\lfloor i/p \rfloor}}, & \textrm{if } 0 \leq i < p(k+1); \\
0 , & \textrm{otherwise.}
\end{cases}
\]
\end{prop}

The following definition of relative $\Spin^c$ structures originates from Turaev's work \cite{Tu90}, but in the current phrasing comes from \cite{Ju06}.  

Fix a Riemannian metric on $(M,\ga)$.  Let $v_0$ denote a nonsingular vector field on $\bdd M$ that points  into $M$ on $R_-(\ga)$ and out of $M$ on $R_+(\ga)$, and that is equal to the gradient of the height function $s(\ga) \times I \to I$ on $\ga$. The space of such vector fields is contractible.

A relative $\Spin^c$ structure is defined to be a {\it homology class} of vector fields $v$ on $M$ such that $v|_{\bdd M}$ is equal to $v_0$.  Here two vector fields $v$ and $w$ are said to be {\it homologous} if there exists an open ball $B \subset \Int(M)$ such that $v$ and $w$ are homotopic on $M \setminus B$ relative to the boundary.  There is a free and transitive action of $H_1(M)$ on $\Spin^c(M,\ga)$ given by {\it Reeb turbulization} \cite[p.\,639]{Tu90}.  This action makes the set $\Spin^c(M,\ga)$ into an $H_1(M)$-torsor.  From now on, we call a map $\iota \colon \Spin^c(M,\ga) \to H_1(M)$ an  {\it affine isomorphism} if $\iota$ is an $H_1(M)$-equivariant bijection.  Note that $\iota$ is completely defined by which element $\fs \in \Spin^c(M,\ga)$ it sends to $0 \in H_1(M)$ (or any other fixed element of $H_1(M)$).

The perpendicular two-plane field $v_0^\perp$ is trivial on $\bdd M$ if and only if $(M,\ga)$ is strongly balanced \cite[Prop.\,3.4]{Ju08}.  Suppose that $(M,\ga)$ is strongly balanced.  Let $t$ be a trivialisation of $v_0^\perp$.  Then there is a map dependent on the choice of trivialisation,
\[
c_1(\cdot, t) \colon \Spin^c(M,\ga) \to H^2(M,\bdd M),
\]
where $c_1(\fs,t)$ is defined to be the relative Euler class of the vector bundle $v^\perp \to M$ with respect to a partial section coming from a trivialisation $t$.  So $c_1(\fs,t)$ is the first obstruction to extending the trivialisation $t$ of $v_0^\perp$ to a trivialisation of $v^\perp$.  Here $v$ is a vector field on $M$ representing the homology class $\fs$.

We now have all the ingredients required to define the sutured Floer polytope.  Let $S(M,\ga)$ be the {\it support} of the sutured Floer homology of $(M,\ga)$.  That is,
\[
S(M,\ga):=\{ \fs \in \Spin^c(M,\ga) \colon SFH(M,\ga,\fs)\neq 0\}.
\]
Consider the map $i \colon H^2(M,\bdd M;\bZ) \to H^2(M,\bdd M;\bR)$ induced by the inclusion $\bZ \hookrightarrow \bR$.  For $t$ a trivialisation of $v_0^\perp$, define
\[
C(M,\ga,t):=\{ i \circ c_1(\fs,t) : \fs \in S(M,\ga)\} \subset H^2(M,\bdd M;\bR).
\]
Then the {\it sutured Floer polytope} $P(M,\ga,t)$ with respect to $t$ is defined to be the convex hull of $C(M,\ga,t)$.  Finally, we have that $c_1(\fs,t_1)-c_1(\fs,t_2)$ is an element of $H^2(M,\bdd M)$ dependent only on the trivialisations $t_1$ and $t_2$ \cite[Lem.\,3.11]{Ju10a}, and therefore we may write $P(M,\ga)$ to mean the polytope in $H^2(M,\bdd M;\bR)$ up to translation.

\begin{rmk} \label{polytope rmk} It is important to note that $c_1$ ``doubles the distances.'' Namely, the map $PD \circ c_1 \colon \Spin^c(M,\ga) \to H_1(M)$ is equal to $2 \iota \colon \Spin^c(M,\ga) \to H_1(M)$, where $\iota$ is an affine isomorphism \cite[5.3.1\,Thm]{Tu90}.  Thus, we can compare two polytopes by comparing the ratios of their side lengths, since the ratios remain the same under affine isomorphisms and doubling.
\end{rmk}

%%%%%%%%%%%%%%%%%%%%%%%%%%%%
\subsection{Sutured torsion}
Each of the groups $SFH(M,\ga,\fs)$ has a relative $\bZ_2$ grading, which is made into an absolute $\bZ_2$ grading by choosing an orientation $\om$ of the vector space $H_*(M,R_-(\ga);\bR)$.  Then,  for every relative $\Spin^c$ structure $\fs$, the Euler characteristic \linebreak $\chi SFH(M,\ga,\fs)$ is well-defined with no sign ambiguity.  Theorem 1 of \cite{FJR10} tells us that the Euler characteristic with respect to the orientation $\om$, denoted by  $\chi SFH(M,\ga,\fs,\om)$, is a function $T_{(M,\ga,\om)} \colon \Spin^c(M,\ga) \to \bZ$ that can be thought of as the maximal abelian torsion of the pair $(M,R_-(\ga))$, in the sense of Turaev \cite{Tu01}.
Fixing an affine isomorphism $\iota \colon \Spin^c(M,\ga) \to H_1(M)$ lets us collect all of these functions into a single generating function
\[
\tau(M,\ga):=\sum_{\fs \in \Spin^c(M,\ga)} T_{(M,\ga,\om)}(\fs) \cdot \iota(\fs).
\]
We refer to $\tau(M,\ga)$ as the {\it sutured torsion} invariant. 

In the case when $(M,\ga)$ is a manifold complementary to a Seifert surface we drop the reference to $\ga$ and write just $\tau(M)$ to mean $\tau(M,\ga)$.  Note that $\tau(M,\ga)$ is an element of the group ring $\bZ[H_1(M)]$, and that it is well-defined up to multiplication by an element of the form $\pm h$, where $h \in H_1(M)$.  We can extend the affine isomorphism $\iota$ linearly to a map on the group rings denoted by the same letter $\iota \colon \bZ[\Spin^c(M,\ga)] \to \bZ[H_1(M)]$. Then
\[
\tau(M,\ga)= \iota(\chi SFH(M,\ga)).
\]

\begin{rmk}
Notice that the abelian group $H_1(M)$ is thought of as a multiplicative group; hence the notion of being well-defined up to multiplication by an element.  Specifically, if $f= \pm h\cdot g$, for elements $f,g$ of the group ring $\bZ[H_1(M)]$, then we use the notation $f \doteq g$.
\end{rmk} 

Finally, let us describe how to compute the torsion $\tau(M,\ga)$ of a given irreducible balanced sutured manifold $(M,\ga)$ with connected subsurfaces $R_\pm(\ga)$. Fix a basepoint $p \in R_-(\ga)$.  Then Proposition 5.1 of \cite{FJR10} tells us how to compute the torsion from the map $\kappa_* \colon \pi_1(R_-(\ga),p) \to \pi_1(M,p)$ induced by the natural inclusion $\kappa \colon R_-(\ga) \hookrightarrow M$.  

First, take a {\it geometrically balanced} presentation of $\pi_1(M,p)$; that is, a presentation 
\[
\pi_1(M,p)=\langle a_1, \ldots, a_m| r_1, \ldots, r_n \rangle,
\]
where the deficiency of the presentation $m-n$ is equal to the genus $g(\bdd M)$ of the boundary of $M$.  

Obtaining a geometrically balanced presentation is not hard. Any balanced sutured manifold  $(M,\ga)$ can be reconstructed in a standard way from a {\it balanced sutured diagram}  $(\Si,\hal,\hbe)$ \cite[Prop.\,2.14]{Ju06}, where $\Si$ is a surface with boundary, and each of $\hal$ and $\hbe$ is a set  containing the same number of pairwise disjoint simple closed curves. To recover $(M,\ga)$, thicken $\Si$ to $\Si \times [0,1]$, regard $\hal$ as curves on $\Si \times \{0\}$, and $\hbe$ as curves on $\Si \times \{1\}$.  Then attach 2-handles along $\hal$ and $\hbe$ to obtain $M$ with sutures $\bdd \Si \times \{1/2\}$.  

Suppose that we picked the orientations so that $R_-(\ga)$ is the component of the boundary on ``the bottom'' that includes the boundaries of the 2-handles attached to $\al$.  Note that the 2-handles attached to $\al$ are precisely the 1-handles attached to $R_-(\ga)$.  Then the generators of the free group $\pi_1(R_-(\ga),p)$ and the cores of the 1-handles attached to $R_-(\ga)$ are a generating set for $\pi_1(M,p)$; the cores of the 2-handles attached to $\hbe$ give the relations of $\pi_1(M,p)$ in these generators.  Therefore, the deficiency of this presentation is equal to the number of generators of $\pi_1(R_-(\ga),p)$: say this number is $l$. Finally, as $M$ is balanced, $l$ is precisely equal to the genus of $\bdd M$. 

Let $\pi_1(R_-(\ga),p):=\langle \si_1, \ldots, \si_l \rangle$.  Then the images of $\si_j$ under the map $\kappa_*$ are words in the generators $a_i$ of $\pi_1(M,p)$.  In later sections, we abuse notation and refer to $\kappa_*(\si_j)$ as $\si_j$.  Now we can form the square matrix of Fox derivatives
\[
\Theta_M:=
\begin{pmatrix}
\varphi \Bigl( \frac{\bdd \kappa_*(\si_j)}{\bdd a_i}\Bigr) &  \varphi \bigl( \frac{\bdd r_k} {\bdd a_i} \bigr)
\end{pmatrix},
\]
where $\varphi \colon \bZ[\pi_1(M,p)] \to \bZ[H_1(M)]$ is the map induced by the abelianization of the fundamental group.
\begin{rmk}
We use the convention that the Fox derivative is computed left-to-right.  For example, take words $u,w \in \bZ[\pi_1(M,p)]$ and apply  the Fox derivative $\frac{\bdd }{\bdd a_i} \colon \bZ [\pi_1(M,p)] \to \bZ[\pi_1(M,p)]$ to $u w$.  Then 
\[
\frac{\bdd(u w) }{\bdd a_i}=\frac{\bdd u} {\bdd a_i} \mbox{aug}(w) + u \frac{\bdd w}{\bdd a_i},
\]
where $\mbox{aug} \colon \bZ[\pi_1(M,p)] \to \bZ$ is the augmentation map.
\end{rmk}

\begin{prop} \label{prop}  \cite[Prop.\,5.1]{FJR10} Let $(M,\ga)$ be a balanced sutured manifold such that $M$ is irreducible and the subsurfaces $R_\pm(\ga)$ are connected.    Then
\[
\tau(M,\ga) \doteq \det \Theta_{M}.
\]
\end{prop}

In particular, Proposition \ref{prop} can be applied in the case of a sutured manifold complementary to a minimal genus Seifert surface of a knot in $S^3$.

Lastly, let us say what it means for two sutured torsion polynomials $\tau_1:=\tau(M_1,\ga_1) \in \bZ[H_1(M_1)]$ and $\tau_2:=\tau(M_2,\ga_2) \in \bZ[H_1(M_2)]$ to be equivalent.  Note that the only relevant choices that we have made is that of the affine isomorphism $\iota_i \colon \Spin^c(M_i,\ga_i) \to H_1(M_i)$, for $i=1,2$.  Therefore, the two sutured torsion polynomials are {\it equivalent} $\tau_1 \sim \tau_2$ if there is an affine isomorphism $\psi \colon H_1(M_1) \to H_1(M_2)$, which extends linearly to a map on the group rings, such that $\psi(\tau_1)\doteq \tau_2$.  Also, we say that $\chi SFH(M_1,\ga_1)$ is {\it equivalent} to $\chi SFH(M_2,\ga_2)$ if $\tau_1 \sim \tau_2$.

%%%%%%%%%%%%%%%%%%%%%%%%%%%%
%%%%%%%%%%%%%%%%%%%%%%%%%%%%
%%%%%%%%%%%%%%%%%%%%%%%%%%%%
\section{The example} \label{example}

Lyon's paper \cite{Lyon} is part of a series of papers in the 70's that aimed to produce examples of knots with nonisotopic Seifert surfaces.  The first few papers by Alford, Schaufele, and Daigle \cite{Alford, AS,Daigle} all give various infinite families of such examples.  Some of these families have readily computable sutured torsion invariants, and it turns out that the sutured torsion does not distinguish between Seifert surfaces in these cases.  However, as we will see in this section, the examples in Lyon's paper can be distinguished by their sutured torsion.

%%%%%%%%%%%%%%%%%%%%%%%%%%%%
\subsection{The knots}
The following construction is taken from \cite[pp.\,1--2]{Lyon}.  Let $k$ be the $(3,4)$ torus knot on the torus $T$.  Let $A$ be a tubular neighbourhood of $k$ on $T$, depicted on Figure \ref{Lyons figure}.  Denote by $A'$ the closure of the complement $T \setminus A$.  The boundary of $A$ has two components; connect these components via the boundary of the twisted strip $B$ as shown in Figure \ref{Lyons figure}.  Define the knot $K$ to be the boundary of $A \cup B$. Note that we can introduce full twists in the strip $B$ to produce an infinite family of knots $K_n$, labelled by the integers, where the strip $B$ of the knot $K_n$ has $2n+1$ half twists. Then Figure \ref{Lyons figure} depicts $K:=K_0$ with one positive half-twist.  The Alexander polynomial of $K_n$ is easily computed to be
\[
\De_{K_n}(t)=(6+12n)t -(11+24n)+ (6+12n)t^{-1}.
\]
Therefore, each knot $K_n$ is nontrivial.  For computational convenience we work with $n \geq -1$, but of course similar computations can be performed for $n<-1$.  The knot $K_{-1}$ is the one for which we are able to show the polytopes statement from Theorem \ref{thm}.

%%%%%%%%%%%%%%%%%%%%%%%%%%%%
\subsection{The Seifert surfaces} Fix a basepoint $p \in K_n$, as in Figure \ref{Lyons figure}.  Observe that $K_n$ bounds two Seifert surfaces $S_n:=A \cup B$ and $S_n':=A' \cup B$; Figure \ref{Seifert surfaces} depicts $S_0$ and $S'_0$.  Let $(Y_n,\ga_n)$ and $(Y_n',\ga'_n)$ be the sutured manifolds complementary to $S_n$ and $S_n'$, respectively.  Note that in both cases $p$ is contained in $K_n$, or more precisely, $p$ is contained in the sutures $s(\ga_n)$ and $s(\ga'_n)$.   From now on we fix an integer $n \geq -1$.  For the remainder of this subsection we drop `$n$' from the subscript in order to avoid cluttered notation.
\begin{figure}[h]
\centering
\includegraphics [scale=0.5]{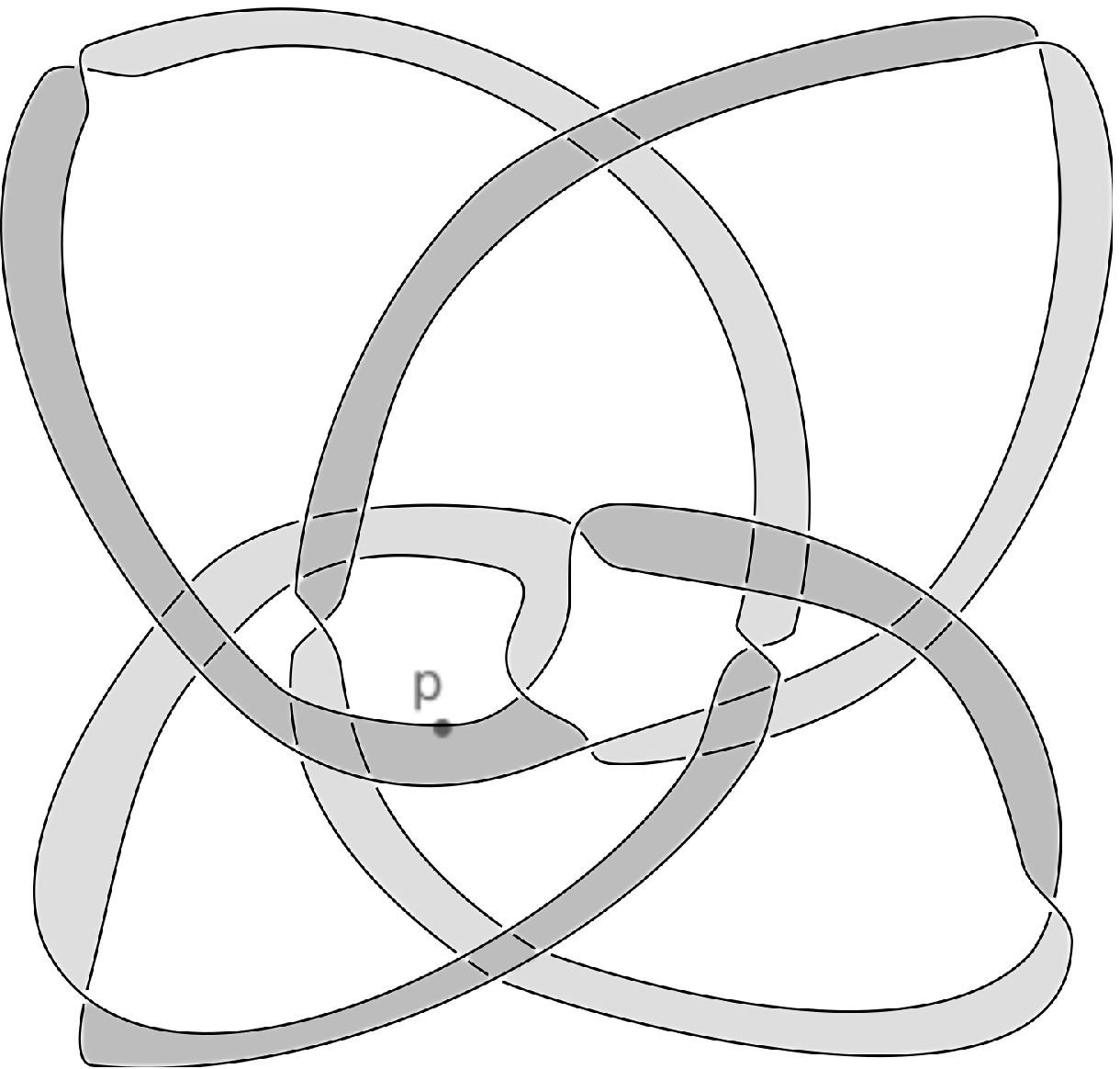}
\hspace{0.5cm}
\includegraphics [scale=0.5]{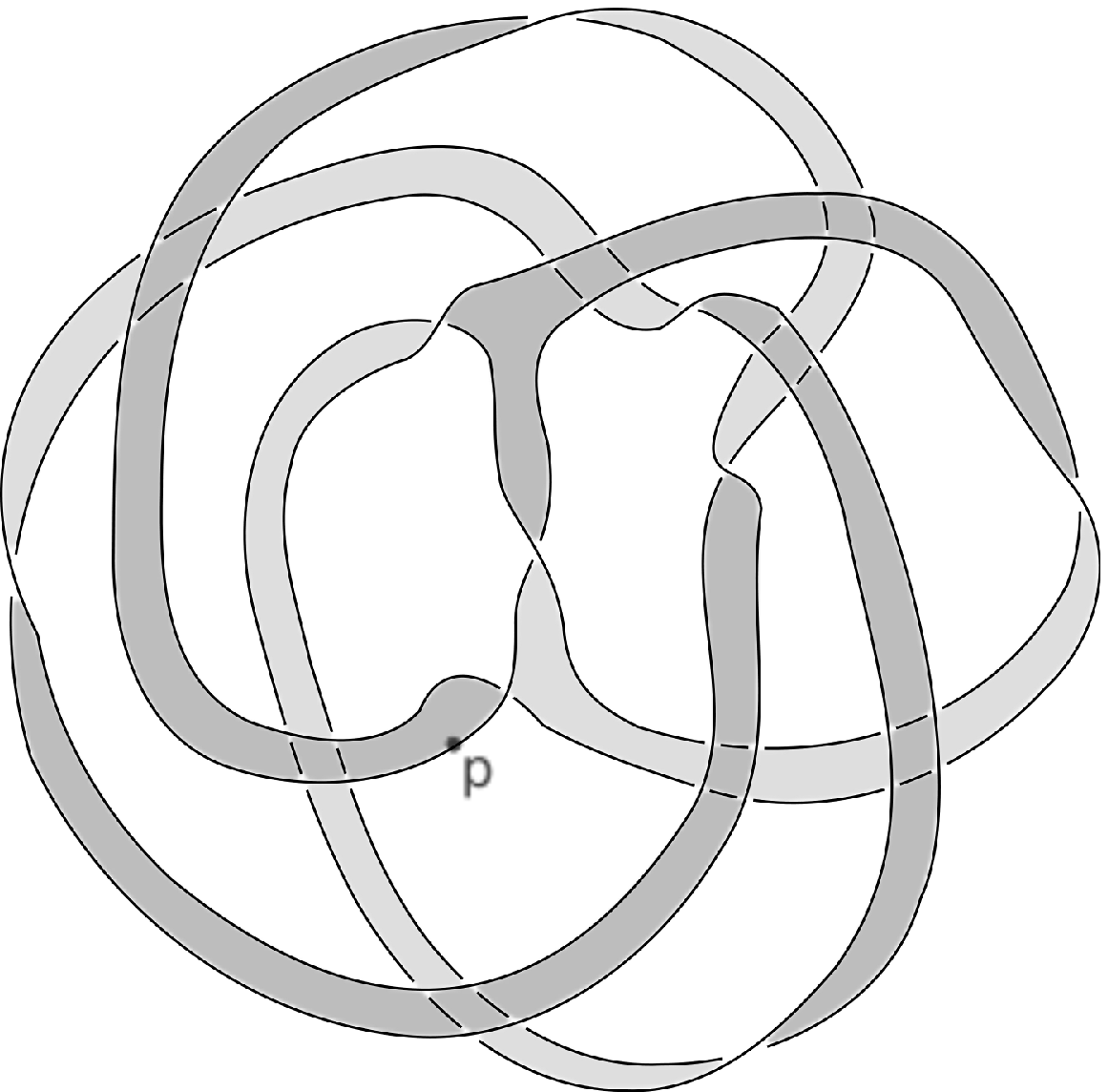}
\caption{The two Seifert surfaces $S_0$ (left) and $S'_0$ (right) for $K_0$.}
\label{Seifert surfaces}
\end{figure}

The torus $T$ gives a genus one Heegaard splitting of $S^3$ into solid tori $U$ and $V$, with $B \subset V$.  This splitting is convenient for computing the fundamental groups $\pi_1(Y,p)$ and $\pi_1(Y',p)$.   From now on, let $V \setminus B$ and $U \setminus A$ stand for the manifolds obtained by removing the appropriate, small (collar) neighbourhoods of $B$ and $A$, respectively. Observe that $V\setminus B$ is a genus two handlebody; let $a$ and $b$ be a generating set of $\pi_1(V\setminus B,p)$ as shown in Figure \ref{Lyoncomplement} (left).  Let $x$ be the generator of $\pi_1(U,p)$, as shown in the same figure.  Figure \ref{Lyoncomplement} (right) shows the discs $D_a$ and $D_b$ that are dual to $a$ and $b$, respectively.  In the remainder of the paper, we compute the homotopy class of a curve in $V \setminus B$ by counting the signed intersections of that curve with the dual discs.
\begin{figure}[h]
\centering
\includegraphics [scale=0.45]{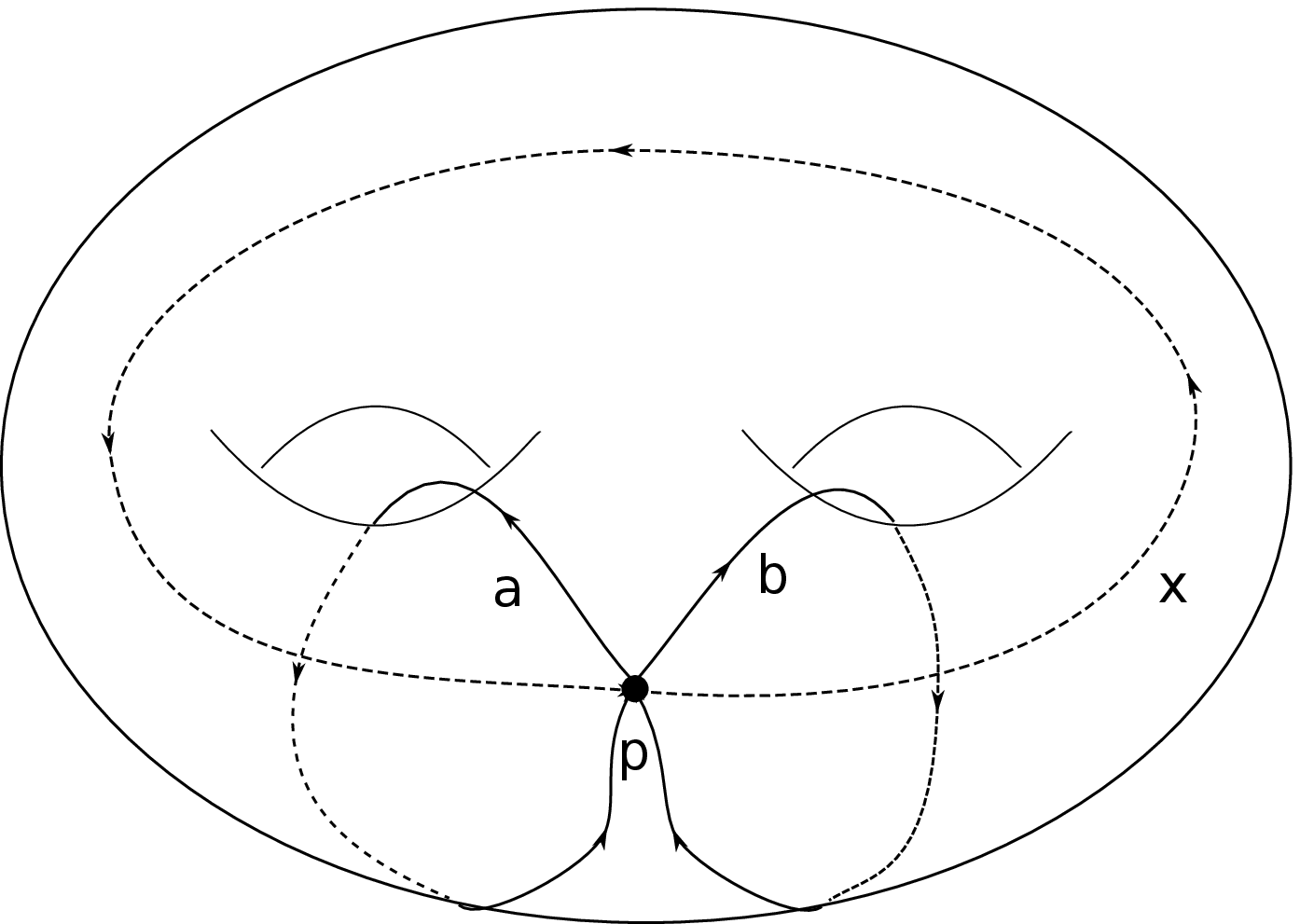}
\hspace{0.5cm}
\includegraphics [scale=0.45]{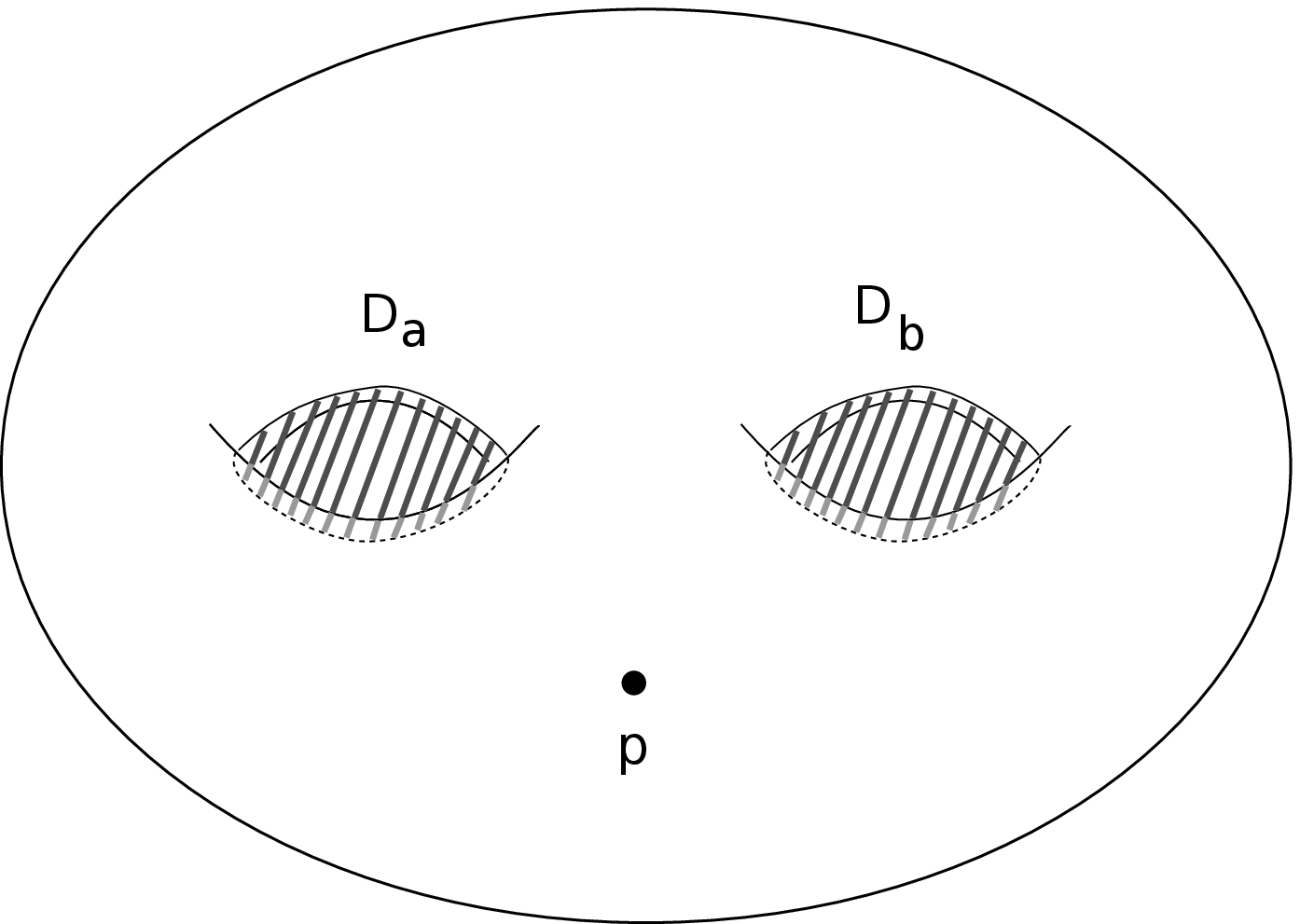}
\caption{Left: The curves $a,b$, and $x$ in the manifolds $Y$ and $Y'$. Right: The dual discs $D_a$ and $D_b$.}
\label{Lyoncomplement}
\end{figure}

In order to compute Fox derivatives, we need to know the fundamental groups of $Y$ and $Y'$.  Note that the following lemma shows that these groups are independent of $n$.
\begin{lemma}
The fundamental groups of the two surface complements have the following presentations:
\begin{align*}
\pi_1(Y,p) \hspace{0.17cm} &=\langle a,b,x| x^3=a^2b^2 \rangle, \\
\pi_1(Y',p)&= \langle x, b \rangle.
\end{align*}
\end{lemma}

 \begin{proof}
View $Y$ as the union of $V\setminus B$ and $U\setminus A$, and then apply Van Kampen's theorem.   In applying Van Kampen's theorem  the only interesting point is what relations come from the intersection $(V\setminus B)\cap(U\setminus A) \cong A'$.  Figure \ref{complement} (left) tells us that the sole relation is $x^3=a^2b^2$, which can be seen by following around the spine of the annulus $A'$ and counting its signed intersections with the dual discs $D_a$ and $D_b$.  So indeed $\pi_1(Y,p)=\langle a,b,x| x^3=a^2b^2 \rangle$.  

Similarly, when computing $\pi_1(Y',p)$, we are interested in what relations come from the intersection $(V\setminus B) \cap (U\setminus A') \cong A$.  Figure \ref{complement} (right) tells us that there is again a single relation: $x^3=bab^2$.    Since $a=b^{-1}x^3b^{-2}$, it follows that $\pi_1(Y',p)\cong \bZ \langle x \rangle * \bZ\langle b\rangle$ .
  
\end{proof}

\begin{figure}[h]
\centering
\includegraphics [scale=0.45]{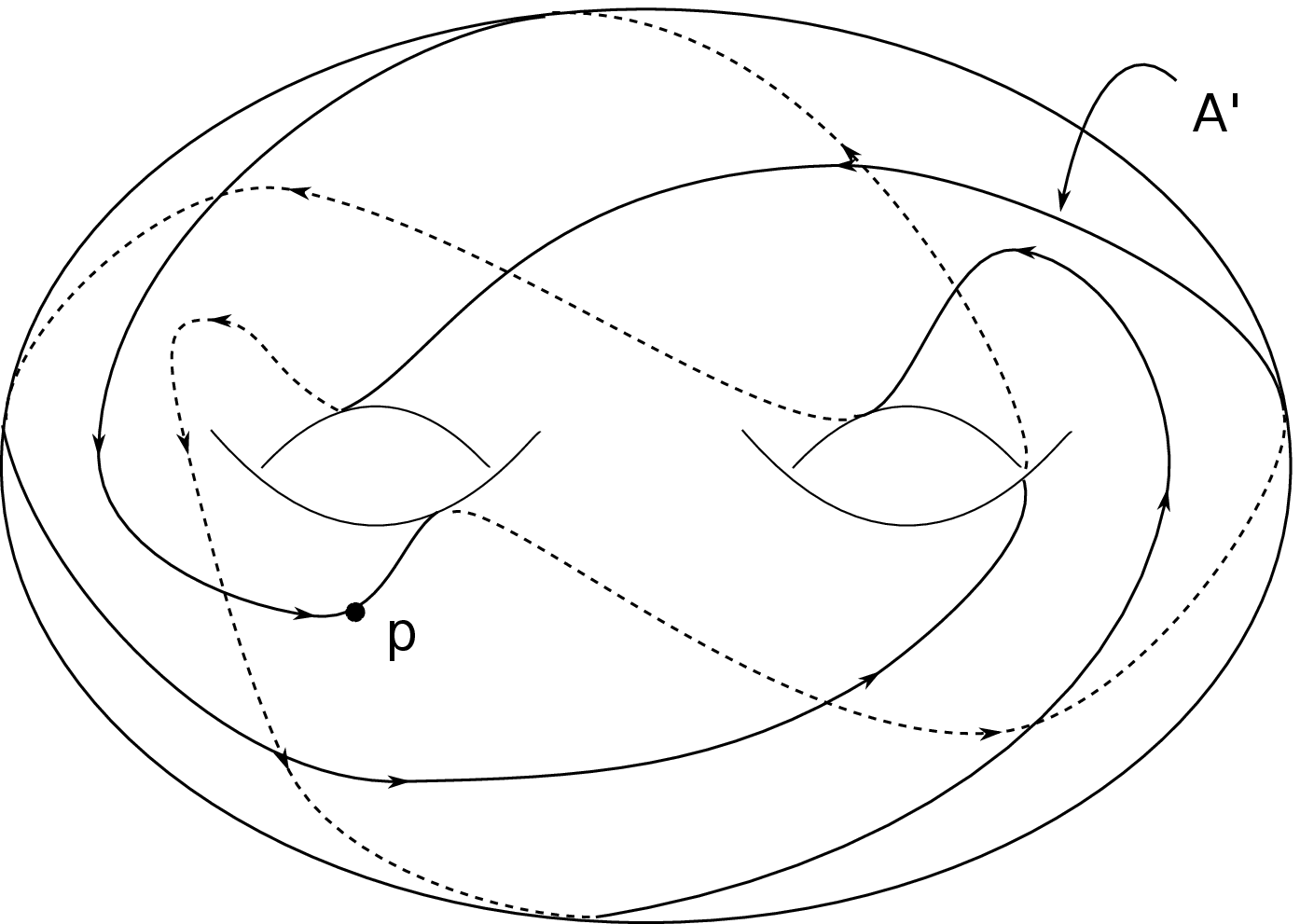}
\hspace{0.5cm}
\includegraphics [scale=0.45]{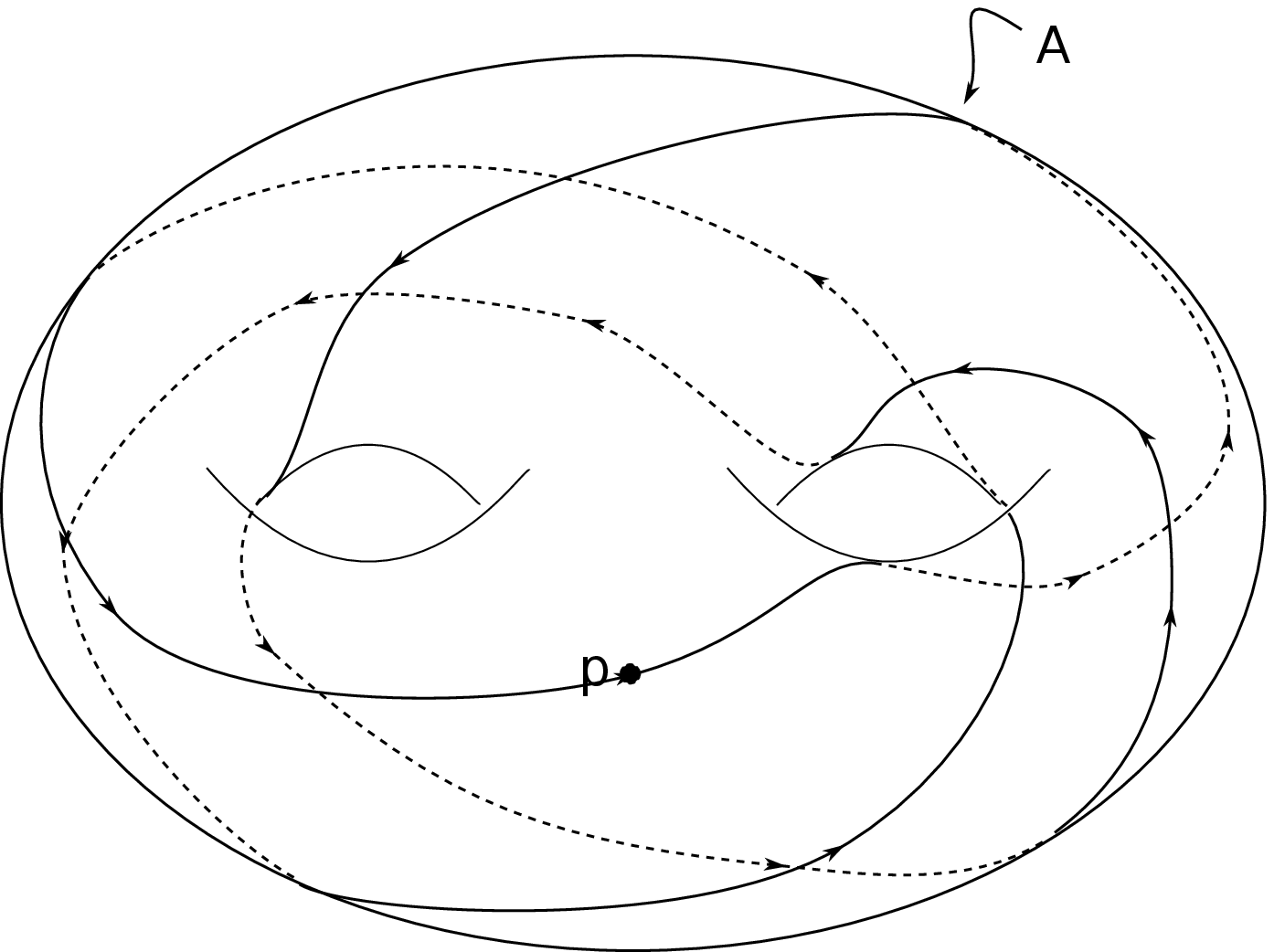}
\caption{Left: spine of $A'$ that gives the relation $x^3=a^2b^2$ in $\pi_1(Y,p)$. Right: spine of $A$ that gives the relation $x^3=bab^2$ in $\pi_1(Y',p)$.}
\label{complement}
\end{figure}  

\begin{rmk} \label{ab} In order to apply Proposition \ref{prop}, we must know explicitly how to abelianize the fundamental groups.  For $\pi_1(Y',p)$, this is clear.  For $\pi_1(Y,p)$, it is convenient to introduce $u:=x^{-1}ab \in \pi_1(Y,p)$. Then, we have $x=u^2$ and $b=u^3a^{-1}$ in homology, so $
H_1(Y;\bZ)\cong \bZ\langle a \rangle \oplus \bZ\langle u \rangle.$
\end{rmk}

\begin{rmk} \label{rmk disjoint} Actually, it can be seen from Figure \ref{Lyons figure} that the surfaces $S$ and $S'$ can be made disjoint in the complement of the knot.  Take two copies of the strip, call them $B$ and $B'$, such that $S=B \cup A$ and $S'=B' \cup A'$. Then $S \cup S'$ form the boundary of a genus-two handlebody, and $S \cap S'=K$.  See Figure \ref{disjoint} for an illustration in the case when $n=0$.  In particular, let $W$ and $X$ be the two handlebodies of the genus two splitting of $S^3$ given by $S \cup S'$, where $W$ is the handlebody on Figure \ref{disjoint} containing the point at infinity. In other words, $W$ can be thought of as $V\setminus B$.  For a particular $n$, note that $W$ and $X$ are sutured manifolds with $K_n$ as their single suture.

The fact that $S$ and $S'$ are disjoint could be used as a shortcut to compute the sutured torsion.  To do so, first compute $\tau(W)$ and $\tau(X)$.  Then use \cite[Prop.\,5.4]{Ju10a} to ``glue'' the two torsion polynomials by Mayer-Vietoris induced maps on the level of homology and so obtain $\tau(Y)$ and $\tau(Y')$. However, we choose not to make use of this shortcut in order to illustrate how Proposition \ref{prop} can be used in a general situation where the two Seifert surfaces are not necessarily disjoint. Therefore, we compute $\tau(Y)$ and $\tau(Y')$ directly from Proposition \ref{prop}, and just point out how $\tau(W)$ and $\tau(X)$ appear in this computation.  See the beginning of subsection \ref{conclusion} for more comments.
\end{rmk}
\begin{figure}[h]
\centering
\includegraphics [scale=0.5]{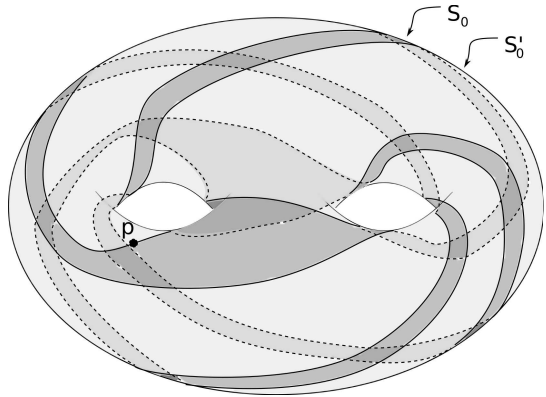}
\caption{The surfaces $S_0$ and $S'_0$ bounding a handlebody of genus two.}
\label{disjoint}
\end{figure}

 In order to specify the $R_\pm$ regions on $(Y,\ga)$ and $(Y',\ga')$, we fix an orientation of the knot and an orientation of $S^3$.  Suppose that these orientations are chosen so that  the union of $R_-(\ga)$ and $R_+(\ga')$ forms the visible side of the genus two surface which is depicted in Figure \ref{disjoint} for the case $n=0$.  
 
Recall that the sutured torsion of a manifold $(M,\ga)$ is defined using the pair of spaces $(M,R_-(\ga))$. Let $\tau^+(M,\ga)$ denote the sutured torsion computed using the same algorithm only with the pair of spaces $(M,R_+(\ga))$.  Fix an affine isomorphism $\iota \colon \Spin^c(M,\ga) \to H_1(M)$. Then, Proposition 2.14 of \cite{FJR10} gives a useful duality result, which says that, as elements of the group ring $\bZ[H_1(M)]$, the two torsion polynomials $\tau(M)$ and $\tau^+(M)$ are equivalent up to a reflection in the origin.  That is, $\tau(M) \doteq \sigma \circ \tau^+(M)$, where $\sigma$ is the linear extension of the inversion map $H_1(M) \to H_1(M)$ given by $h \mapsto h^{-1}$.  

\begin{rmk} \label{R+} In subsection \ref{computing prime}, we compute $\tau^+(Y')$ even though we write $\tau(Y')$.  Once computed, the polynomial $\tau^+(Y')$ is easily seen to be centrally symmetric, so $\tau^+(Y') \doteq \tau(Y')$ and we are justified in writing $\tau(Y')$ instead. 
\end{rmk}

%%%%%%%%%%%%%%%%%%%%%%%%%%%%%%%%%%%
\subsection{Computing $\tau(Y_n)$}
Take $\al$ and $\be$ to be the generators of $\pi_1(S_n,p)$ as depicted in Figure \ref{generators of S}.  Push these curves into the complement.  In particular, push them into $V\setminus B$; this operation amounts to considering the inclusion map $\kappa_* \colon \pi_1(R_-(\ga_n),p) \to \pi_1(Y_n,p)$ that occurs in the definition of the matrix $\Theta_{Y_n}$.
\begin{figure}[h]
\centering
\includegraphics [scale=0.45]{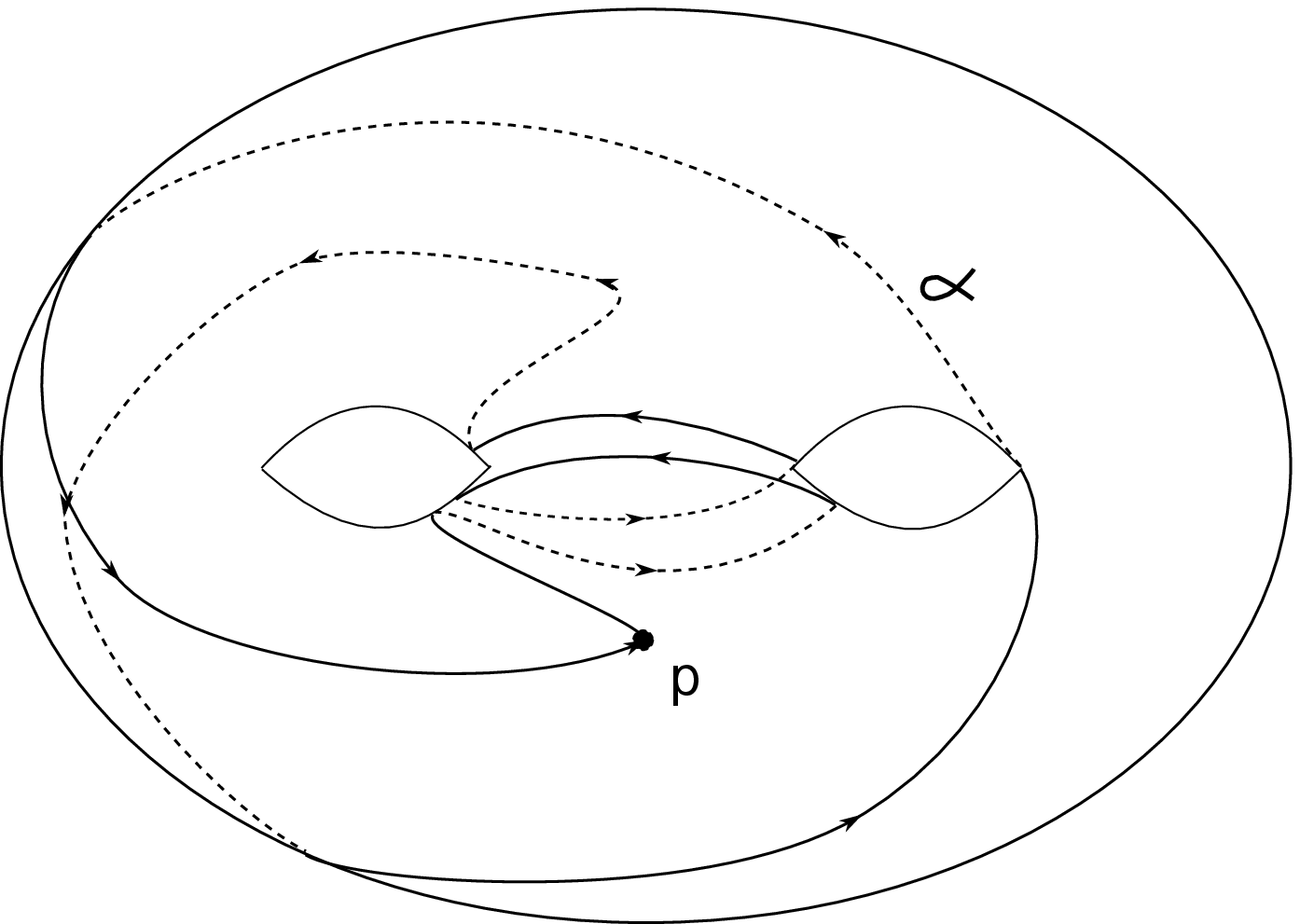}
\hspace{0.5cm}
\includegraphics [scale=0.45]{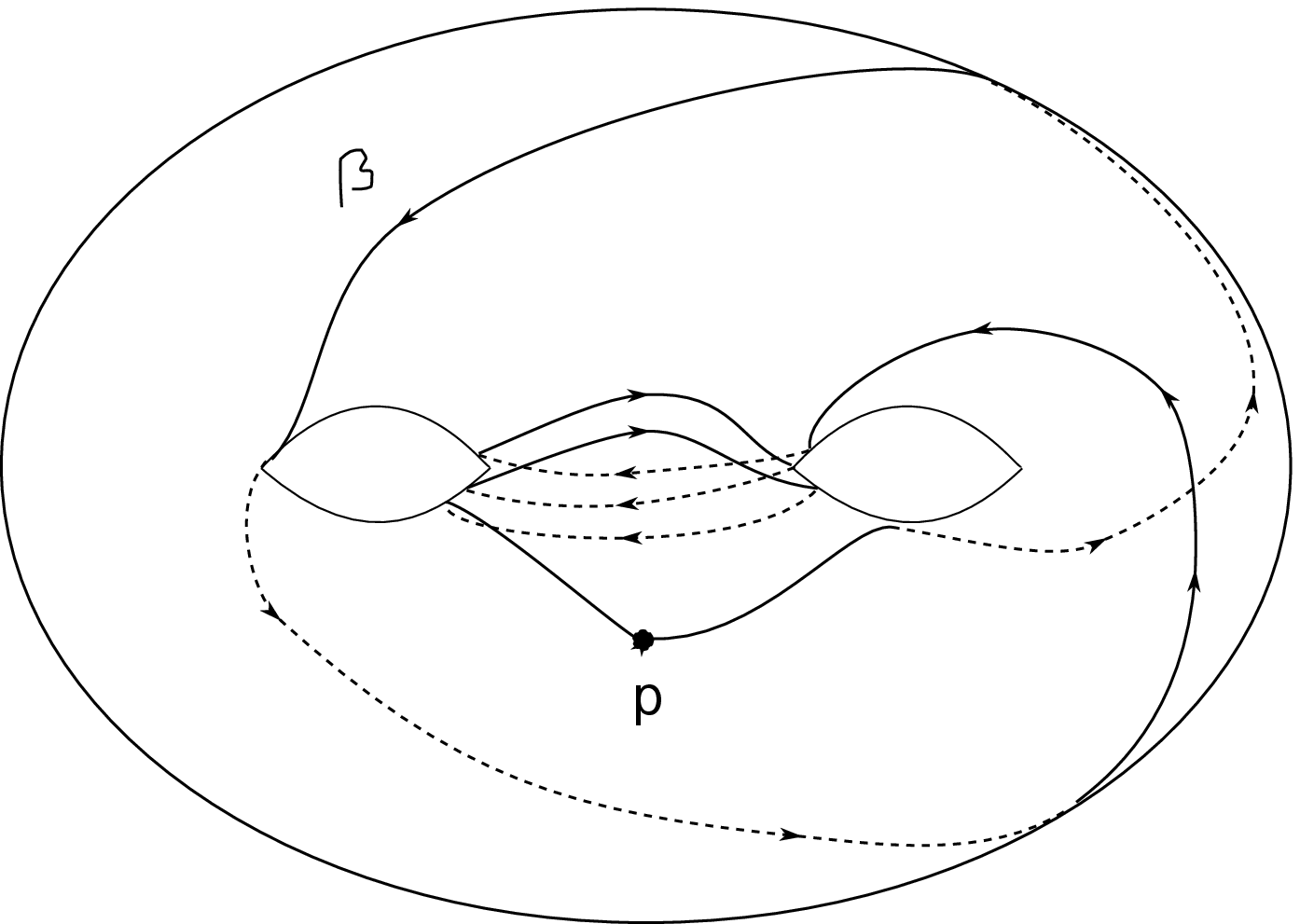}
\caption{The generators $\alpha$ and $\beta$ of $\pi_1(S_2,p)$.}
\label{generators of S}
\end{figure}
Next, read off the relations $\al=a(b^{-1}a)^{n}b$ and $\be=ba(ba^{-1})^{n}ba^{-1}$. So we have
\begin{gather*}
\al=(a b^{-1})^{n+1} b^2, \label{al'} \\ 
\be=ba(ba^{-1})^{n+1}. \label{be'}
\end{gather*}
 
It turns out that $\al$ and $\be$ are curves entirely given in the two generators $a,b$.  Therefore, their Fox derivatives with respect to $x$ are zero.   Denote by $r:=x^3b^{-2}a^{-2}$ the group relation of $\pi_1(Y_n,p)$.  So by Proposition \ref{prop},
\[
\tau(Y_n)\doteq \det \Theta_{Y_n}= \varphi \Bigl( \frac{\bdd r}{\bdd x}\Bigr) \cdot \det
\begin{pmatrix}
\varphi \bigl( \frac{\bdd \al}{\bdd a} \bigr)& \varphi \bigl(  \frac{\bdd \be}{\bdd a}\bigr)\\
\varphi \bigl( \frac{\bdd \al}{\bdd b} \bigr) &\varphi \bigl(   \frac{\bdd \be}{\bdd b}\bigr)
\end{pmatrix}.
\]
We have $\frac{\bdd r}{\bdd x}=1+x+x^2$ and
\begin{align*}
\frac{\bdd \al}{\bdd a} &= \frac{(ab^{-1})^{n+1}-1}{ab^{-1}-1},   &
\frac{\bdd \be}{\bdd a} &=b-baba^{-1}\frac{(ba^{-1})^{n+1}-1}{ba^{-1}-1},\\
\frac{\bdd \al}{\bdd b} &=-ab^{-1}\frac{(ab^{-1})^{n+1}-1}{ab^{-1}-1}+(ab^{-1})^{n+1}(1+b), & 
\frac{\bdd \be}{\bdd b}&= 1+ba\frac{(ba^{-1})^{n+1}-1}{ba^{-1}-1}. 
\end{align*}
Now compute the polynomial 
$
q_n(a,b):=\det \begin{pmatrix}
 \frac{\bdd \al}{\bdd a}&   \frac{\bdd \be}{\bdd a}\\
\frac{\bdd \al}{\bdd b}  &  \frac{\bdd \be}{\bdd b}
\end{pmatrix}
$ as a polynomial in $\bZ[H]$, where $H:=\bZ\langle a \rangle \oplus \bZ\langle b \rangle$. Then 
\begin{align*}
q_n(a,b) &=-\frac{b}{a-b} \left( 1 + a+a b + ab^2 - \left( \frac{a}{b}\right)^{n+1}- b\left( \frac{a}{b}\right)^{n+1} - b^2 \left( \frac{a}{b}\right)^{n+1} - a b^2 \left( \frac{a}{b}\right)^{n+1}\right) \\
   &\doteq \frac{b}{a-b} \left(a^{n+1} (1+b+b^2 +ab^2) - b^{n+1}(1+a+ab+ab^2)\right).
\end{align*}  
This polynomial appears again when we compute $\tau(Y'_n)$; see the beginning of subsection \ref{conclusion} for an explanation.  Note that 
\begin{equation} \label{rec}
q_{n+1}(a,b)\doteq a \cdot q_{n}(a,b)+ b^{n+2} (1+a+ab+ab^2).
\end{equation}
Recall from Remark \ref{ab} how to abelianize $\pi_1(Y_n,p)$.  To obtain the sutured torsion we need to calculate
\begin{equation} \label{torprime}
\tau(Y_n) \doteq \varphi \bigl(q_n(a,b) \cdot (1+x+x^2)\bigr),
\end{equation}
 which yields a polynomial in $\bZ[a^{\pm 1},u^{\pm 1}]$.  For a general $n \geq 0$, we have
 \begin{equation} \label{tor}
\tau(Y_n)\doteq \frac{(1+u^2+u^4)}{a^2-u^3} \left[ (a^2+u^3 a+u^6 a +u^6)a^{2n+2} - u^{3n+3}(a^3+a^2+u^3 a^2+u^6 a)  \right].
\end{equation}

As $q_0(a,b)=1+ab^2$ has all positive coefficients, it follows from \eqref{rec} that all the coefficients of $q_n(a,b)$ are of the same sign.  The recursive equation \eqref{rec} together with \eqref{torprime} implies that the coefficients of $\tau(Y_n)$ add up to $6+12n$, which is exactly  the top term of $\De_{K_n}(t)$, as it should be by Lemma\,6.4 of \cite{FJR10}.

%%%%%%%%%%%%%%%%%%%%%%%%%%%%
\subsection{Computing $\tau(Y'_n)$} \label{computing prime}
We follow a similar procedure to compute the sutured torsion of $Y'_n$.   Take $\al$ and $\be$ to be the generators of $\pi_1(S'_n,p)$ as depicted in Figure \ref{generators of S'}.  As before, push the curves into $V \setminus B$; this operation amounts to considering the inclusion map $\kappa_* \colon \pi_1(R_+(\ga_n'),p) \to \pi_1(Y_n',p)$.  Therefore, what we refer to as $\tau(Y_n')$ below is actually $\tau^+(Y_n')$; see Remark \ref{R+}.
\begin{figure}[h]
\centering
\includegraphics [scale=0.45]{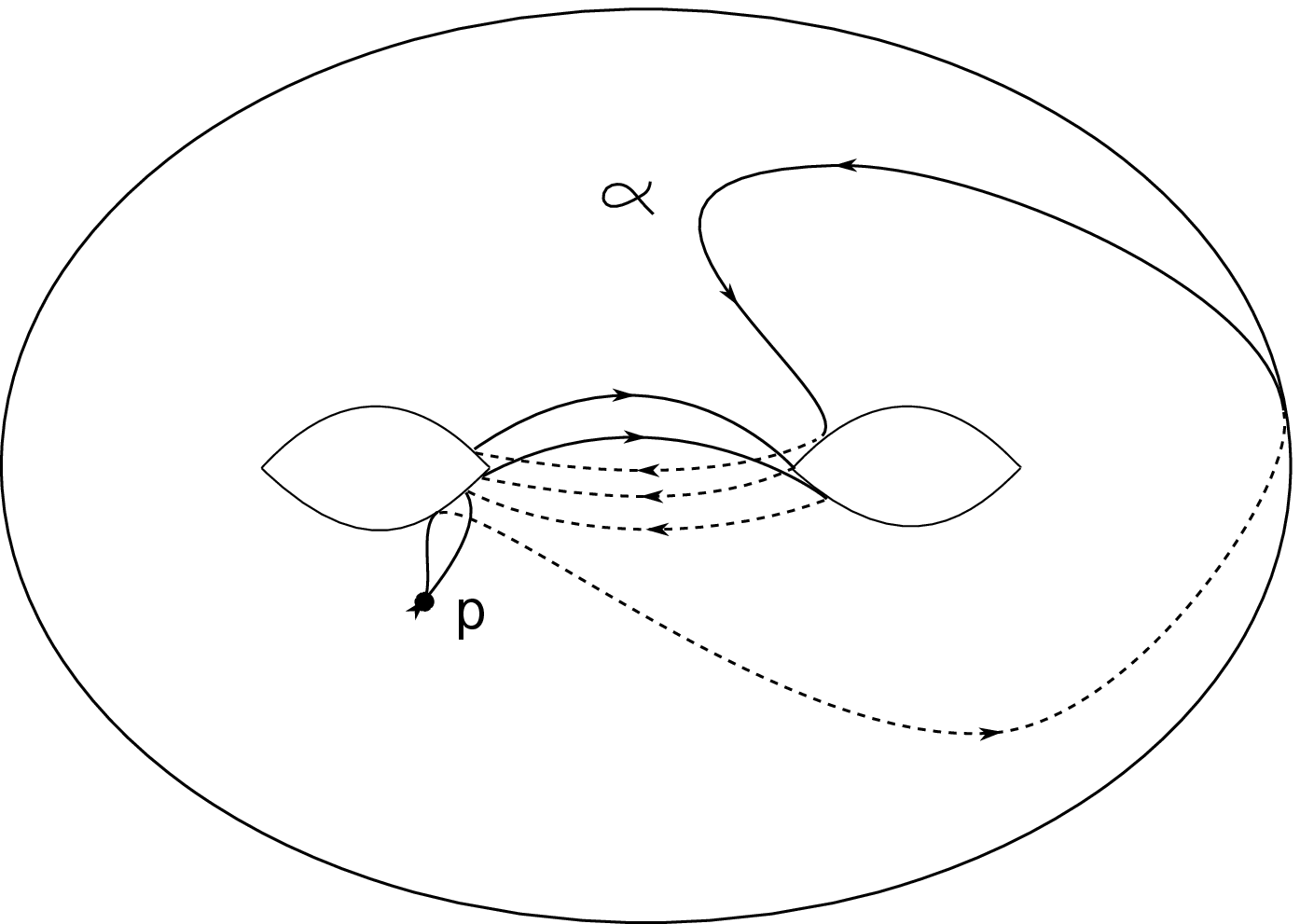}
\hspace{0.5cm}
\includegraphics [scale=0.45]{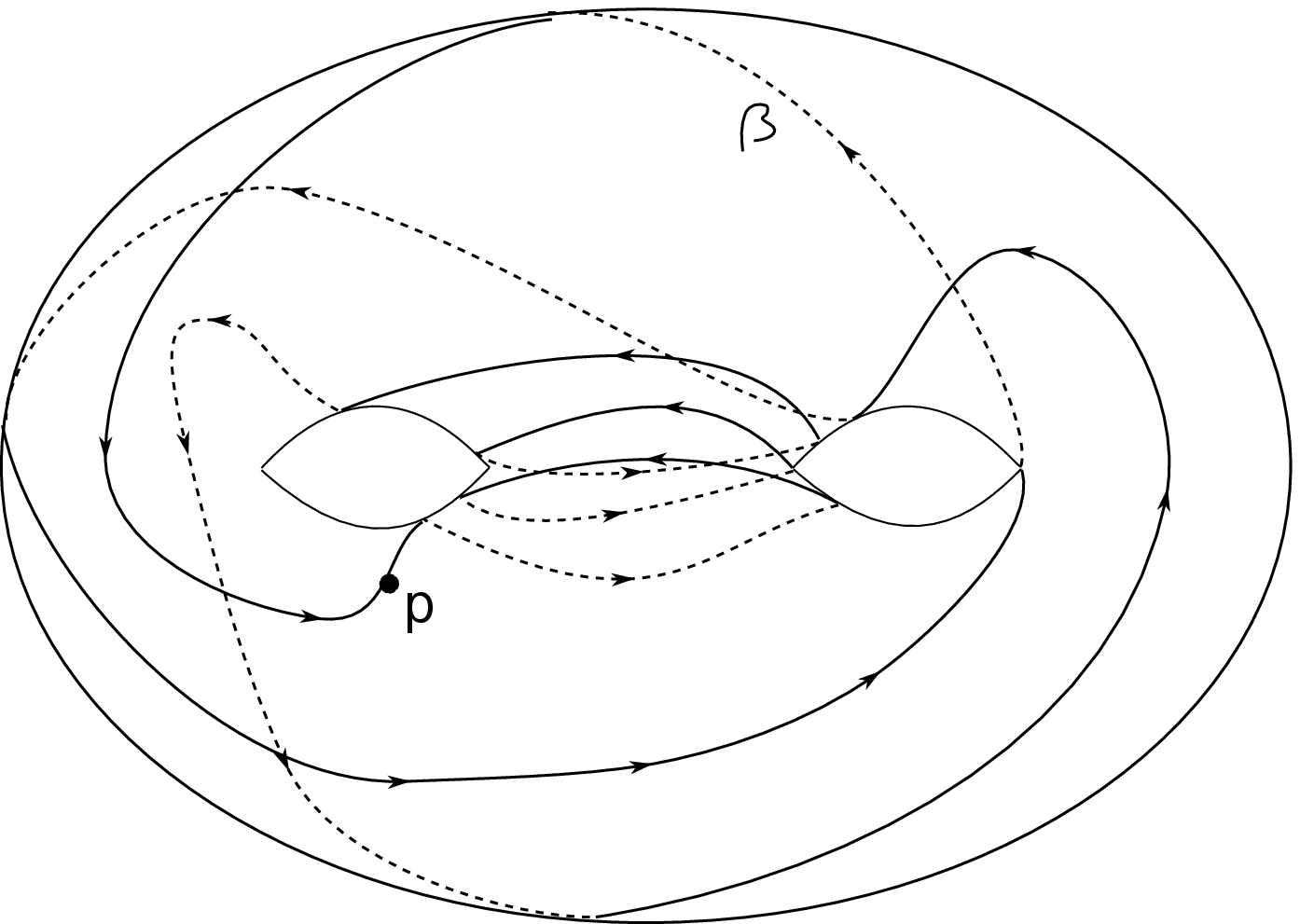}
\caption{The generators $\alpha$ and $\beta$ of $\pi_1(S'_2,p)$.}
\label{generators of S'}
\end{figure}
Read off the relations $\al=ab(a^{-1}b)^{n}a^{-1}$ and \linebreak $\be=ab^{-1}(ab^{-1})^{n}ab^2$.  So we have
\begin{gather*}
\al=a(ba^{-1})^{n+1} \label{al}, \\ 
\be=(ab^{-1})^{n+1}ab^2 \label{be}.
\end{gather*}

Denote by $r:=x^3b^{-2}a^{-1}b^{-1}$ the group relation.  Even though $\pi_1(Y'_n,p)$ is a free group, we choose to compute $\tau(Y'_n)$ in the presentation with three generators and one relation, in order to exhibit similarities with $\tau(Y_n)$.  As before, the Fox derivatives of $\al$ and $\be$ with respect to $x$ are both zero, so again the only relevant Fox derivative of $r$ is $\frac{\bdd r}{\bdd x}=1+x+x^2$.  The other Fox derivatives are:

\begin{align*}
\frac{\bdd \al}{\bdd a} &= 1-aba^{-1} \frac{(ba^{-1})^{n+1}-1}{ba^{-1}-1},  &
\frac{\bdd \be}{\bdd a} &=\frac{(ab^{-1})^{n+1}-1}{ab^{-1}-1} + (ab^{-1})^{n+1}, \\
\frac{\bdd \al}{\bdd b} &=a\frac{(ba^{-1})^{n+1}-1}{ba^{-1}-1}, & 
\frac{\bdd \be}{\bdd b}&=-ab^{-1}\frac{(ab^{-1})^{n+1}-1}{ab^{-1}-1}+(ab^{-1})^{n+1} a(1+b) .
\end{align*}
Computing the polynomial 
$
q_n'(a,b) := \det \begin{pmatrix}
 \frac{\bdd \al}{\bdd a}&   \frac{\bdd \be}{\bdd a}\\
\frac{\bdd \al}{\bdd b}  &  \frac{\bdd \be}{\bdd b}
\end{pmatrix} \in \bZ[H]
$
we find that $q'_n(a,b)\doteq q_n(a,b)$.  Therefore, the difference between the two sutured torsion invariants comes from the abelianization maps.

Recall that $a= x^3b^{-3} \in H_1(Y_n';\bZ)$ and make this substitution for $a$ in the expression 
\[
\tau(Y'_n)\doteq \varphi \bigl( q'_n(a,b) \cdot (1+x+x^2) \bigr),
\]
to find $\tau(Y'_n)$ as a polynomial of $\bZ[b^{\pm 1}, x^{\pm 1}]$. For a general $n \geq 0$, we have
\begin{equation} \label{tor'}
\tau(Y'_n) \doteq \frac{(1+x+x^2)}{x^3-b^4} \left[ x^{3n+3} (b^5+ b^4+b^3+ x^3 b^2)-b^{4n+4}(b^3+x^3 b^2+x^3 b+x^3) \right]. 
\end{equation}
The same argument as before shows that the coefficients of $\tau(Y'_n)$ add up to $6+12n$, as expected.

%%%%%%%%%%%%%%%%%%%%%%%%%%%%
\subsection{Conclusion} \label{conclusion}
The polynomials $q(a,b):\doteq q_n(a,b) \doteq q_n'(a,b)$ and $(1+x+x^2)$ appear in the computations of $\tau(Y_n)$ and $\tau(Y_n')$ .  Indeed, in both cases  the sutured torsion is computed by abelianizing an expression of the form $q(a,b) \cdot (1+x+x^2)$.  With regards to Remark \ref{rmk disjoint} this phenomenon is not surprising.  In particular, from the work we have already done, it is not hard to see that
\begin{align*}
\tau(W) &\doteq q(a,b) \hspace{0.68cm} \in H_1(W)\cong\bZ[a^{\pm1},b^{\pm1}], \\
\tau(X)  &\doteq 1+x+x^2 \in H_1(X).
\end{align*}
For us, these observations are useful inasmuch as they verify our computations.  In general, if the Seifert surfaces are not disjoint, then such a verification is not at our convenience.

\begin{rmk}
Note that we have just shown that two vertices of the {\it Kakimizu complex} \cite{Ka92} of $K_n$ have associated to them different sutured torsions, and hence different sutured Floer homology groups.
\end{rmk}

We claim that the sutured torsion invariants $\tau(Y_n)$ and $\tau(Y_n')$ given in \eqref{tor} and \eqref{tor'} are not equivalent for all $n \geq 0$.  For $n=0$, we have
\begin{align*}
\tau(Y_0) &\doteq(a+u^6)(1+u^2+u^4) \in \bZ[a^{\pm 1},u^{\pm 1}], \\
\tau(Y_0')  &\doteq(b+x^3)(1+x+x^2)  \hspace {0.23cm} \in \bZ[b^{\pm 1},x^{\pm 1}].
\end{align*}
Inspection reveals that there is no affine isomorphism $H_1(Y_0)\to H_1(Y_0')$, taking one sutured torsion polynomial onto the other. See Figure \ref{support 1} for the supports.
\begin{figure}[h]
\centering
\includegraphics [scale=0.45]{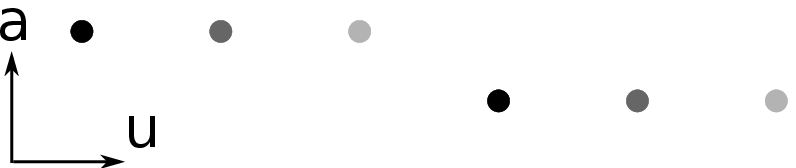} \hspace{2cm}
\includegraphics [scale=0.45]{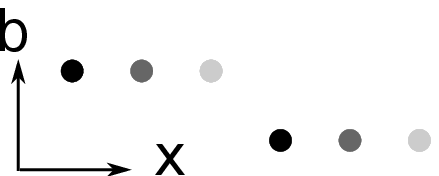}
\caption{Left: the support of $\tau(Y_0)$. Right: the support of $\tau(Y_0')$.}
\label{support 1}
\end{figure}

\begin{rmk}
The three different shades of grey in the support of the polynomials indicate the ``shift'' of $q_n(a,b)$ by $
1+u^2+u^4$ and of $q_n'(a,b)$ by $1+x+x^2$.
\end{rmk}

For $n=1$, the relations are 
\begin{align*}
\tau(Y_1)&\doteq(a^3 + a u^3 + a^2 u^3 + a u^6 + a^2 u^6 + u^9)(1+u^2+u^4) \in \bZ[a^{\pm 1},u^{\pm 1}], \\
\tau(Y_1')&\doteq(b^5 + b x^3 + b^2 x^3 + b^3 x^3 + b^4 x^3 + x^6)(1+x+x^2)  \hspace{0.2cm} \in \bZ[b^{\pm 1},x^{\pm 1}]. 
\end{align*}
Figure \ref{support 2} indicates that the support of $\tau(Y'_1)$ contains a $3 \times 4$ parallelogram, which cannot be found in the support of $\tau(Y_1)$.  Therefore, there too is no affine isomorphism taking one to the other.
\begin{figure}[h]
\centering
\includegraphics [scale=0.45]{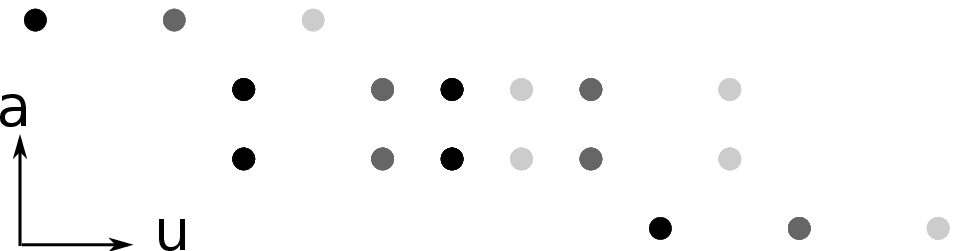} \hspace{2cm}
\includegraphics [scale=0.45]{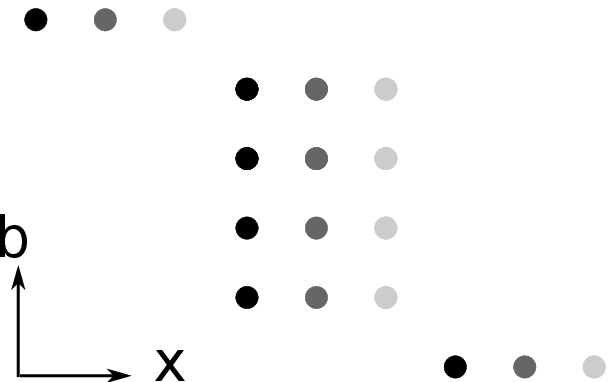}
\caption{Left: the support of $\tau(Y_1)$.  Right: the support of  $\tau(Y_1')$.}
\label{support 2}
\end{figure}

Lastly, for a general $n > 0$, the supports of the sutured torsion follows the pattern from $n=1$, with another parallelogram containing twelve points being added for each increase of $n$ by one; see Figure \ref{support n}.  The same argument as for $n=1$ shows that there is no affine isomorphism taking one torsion polynomial onto another, and thus $\tau(Y_n') \not \sim \tau(Y_n')$. 

\begin{rmk}
For $n>0$, observe that the convex hulls of the supports in both cases are hexagons, only with sides of different length.  For $\tau(Y_n)$ the sides of the convex hull are of slope $-2/3,-1/6,0$ and length $n,1,4$, respectively.  On the other hand, for $\tau(Y'_n)$ the sides of the convex hull are of slope $-4/3,-1/3,0$ and length $n,1,2$, respectively.  So alternatively, we can argue that no affine isomorphism taking one convex hull onto the other.  For $n=-1$, see the latter part of the proof of Theorem \ref{thm}.  For $n<-1$, the sutured torsion invariants can be computed similarly, and an analogous argument can be made to show that they are nonequivalent.
\end{rmk} 

\begin{figure}[h]
\centering
\includegraphics [scale=0.45]{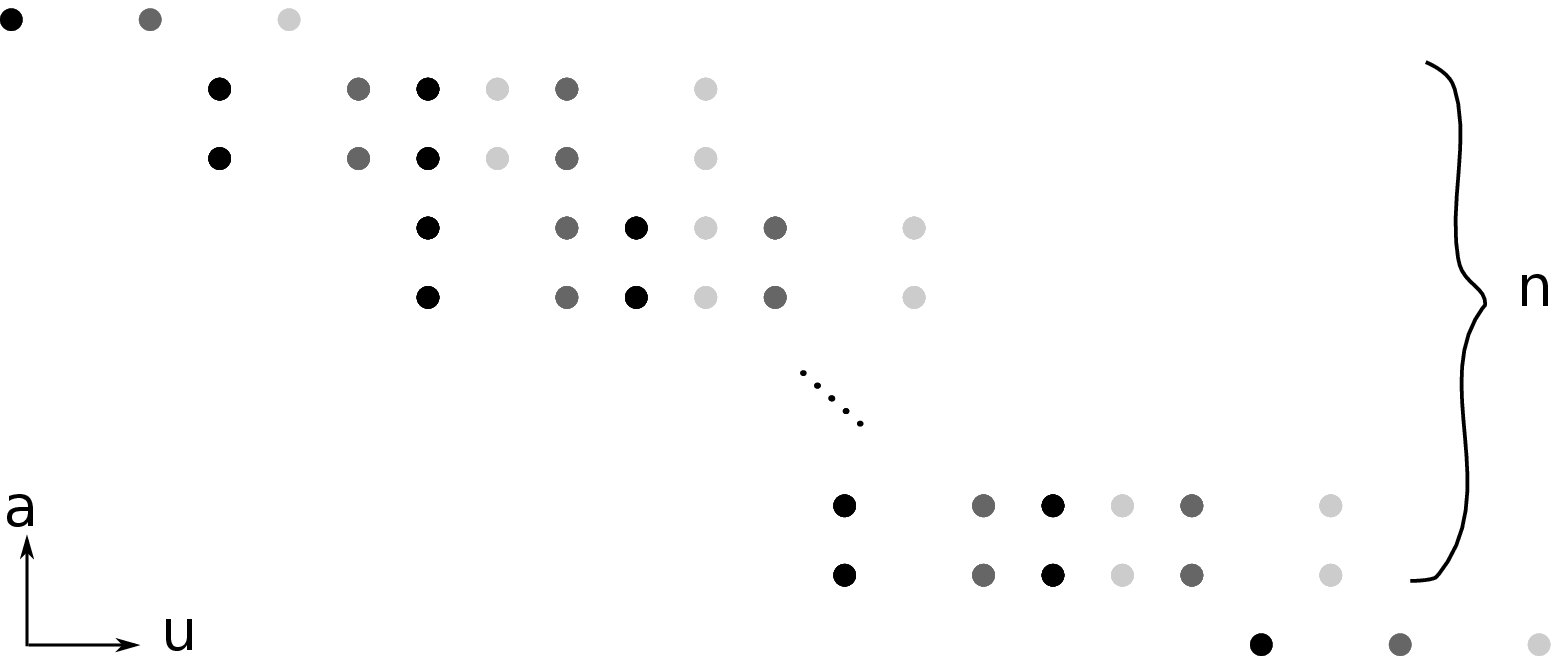} \hspace{1.5cm}
\includegraphics [scale=0.45]{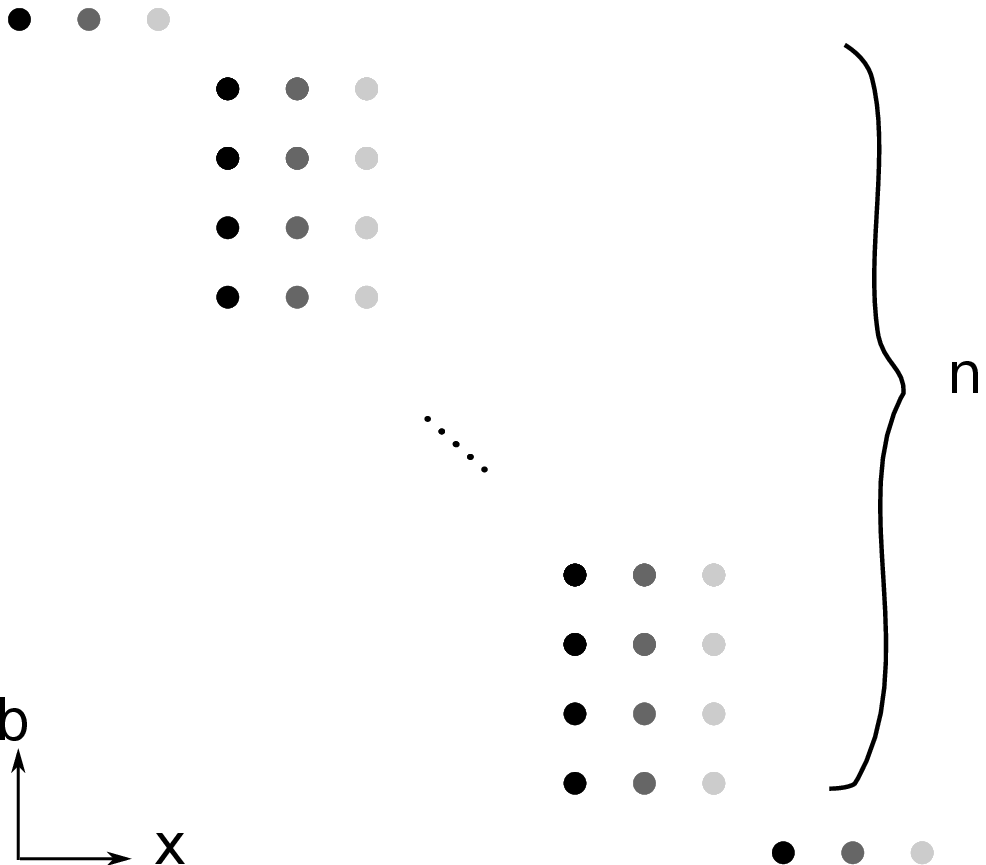}
\caption{Left: the support of $\tau(Y_n)$.  Right: the support of  $\tau(Y_n')$.}
\label{support n}
\end{figure}

\begin{proof} [Proof of Theorem \ref{thm}]
Let $K:=K_0$, and set $R_1:=S_0$ and $R_2:=S_0'$.  Then $\tau(S^3(R_1)) \not \sim \tau(S^3(R_2))$.  Therefore, $SFH(S^3(R_1)) \not \cong SFH(S^3(R_2))$ as $\Spin^c$-graded groups.  For any $n>0$, the knots $K_n$ together with pairs of minimal genus Seifert surfaces $(S_n,S_n')$ have the same property.

For notational convenience, in the remainder of the proof we suppress any references to the sutures of manifolds, as they are clearly understood. 

To prove the statement about polytopes, consider the knot $K_{-1}$. 
\begin{figure}[h]
\centering
\includegraphics [scale=0.5]{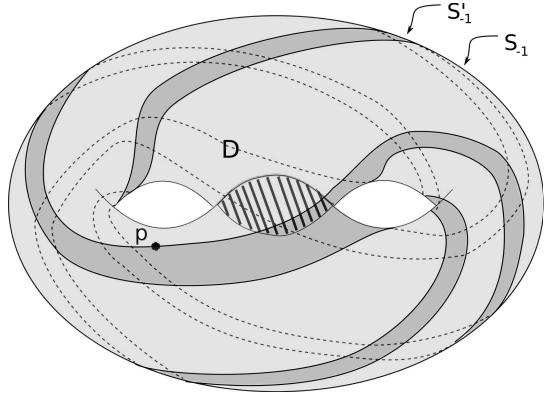}
\caption{The surfaces $S_{-1}$ and $S'_{-1}$, together with a decomposing disc $D$ for $X$.}
\label{disjoint -1}
\end{figure}

See  Figure \ref{disjoint -1} for an analogue of Figure \ref{disjoint} in the case of $K_{-1}$.  Since our torsion computations hold for $n=-1$, we easily compute that $q_{-1}(a,b)\doteq 1+b$, and that the two sutured torsion polynomials are given by the following polynomials:
\begin{align*}
\tau(Y_{-1}) & \doteq (a+u^3)(1+u^2+u^4), \\
\tau(Y'_{-1}) &\doteq (1+b)(1+x+x^2).
\end{align*}

Next, observe that the disc $D_a$ from Figure \ref{Lyoncomplement} (right) gives a product decompositions of $W_{-1}$.   Similarly, the disc $D$ in Figure \ref{disjoint -1} gives a product decomposition of $X_{-1}$.  In particular, the two handlebodies are product decomposed into solid tori with two sutures each;  in the notation of Proposition \ref{torus prop}, we have
\begin{gather*}
W_{-1} \leadsto^{D_a} T(2,1;2), \\
X_{-1} \leadsto^{D} T(3,4;2).
\end{gather*}
Lemma \ref{product decomp} and Proposition \ref{torus prop} imply that 
\begin{align*}
SFH(W_{-1}) &=SFH(T(2,1;2)) =\bZ^2, \\
SFH(X_{-1})&=SFH(T(3,4;2)) =\bZ^3.
\end{align*}
Juh\'asz's {\it decomposition formula} \cite[Prop.\,8.6]{Ju08} implies that $SFH(Y_{-1})$ and \linebreak $SFH(Y'_{-1})$ are isomorphic to $\bZ^6$.  Hence, the sutured torsion $\tau(Y_{-1})$ is the image of the support $S(Y_{-1})$ under some affine isomorphism $\iota \colon \Spin^c(Y_{-1}) \to H_1(Y_{-1};\bZ)$.  Similarly for $\tau(Y_{-1}')$.  By Remark \ref{polytope rmk}, it follows that we can now easily compare the polytopes: Figure \ref{support -1} shows that $P(Y_{-1})$ is a parallelogram with ratio of side lengths 1:4, whereas $P(Y'_{-1})$ is a parallelogram with ratio of side lengths 1:2.  Therefore, the polytopes are different.
\begin{figure}[h]
\centering
\includegraphics [scale=0.45]{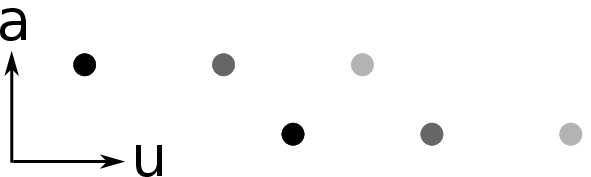} \hspace{2cm}
\includegraphics [scale=0.45]{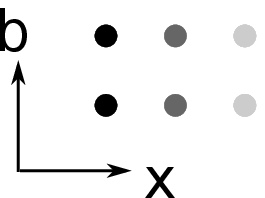}
\caption{Left: the support of $\tau(Y_{-1})$. Right: the support of $\tau(Y_{-1}')$.}
\label{support -1}
\end{figure}

\end{proof}

\begin{rmk}
We see from the sutured torsion polynomials $\tau(Y_{-1})$ and $\tau(Y_{-1}')$ that \linebreak $SFH(Y_{-1},\ga_{-1})$ and $SFH(Y_{-1}',\ga_{-1}')$ are supported in a single $\bZ/2$ homological grading, and we know that both groups are torsion-free.  Therefore, the proof of Theorem \ref{thm} shows that the sutured torsion and the sutured Floer polytope can distinguish between Seifert surfaces whose complementary manifolds are {\it sutured L-spaces} \cite[Def.\,1.1]{FJR10}.
\end{rmk}

%------ BIBLIOGRAPHY ---------

%%%%%%%%%%%%%%%%%%%%%%%%%%%%%%%%%%%

\vspace{1cm}

{\sc \noindent Mathematics Institute, Zeeman Building, University of Warwick, UK.}
\\
E-mail address: i.altman@warwick.ac.uk

\end{document}